\documentclass{amsart}

\usepackage[utf8]{inputenc}
\usepackage{amsfonts}
\usepackage{amsmath}
\usepackage{amssymb}
\usepackage[english]{babel}
\usepackage{hyperref}

\newtheorem{theorem}{Theorem}
\newtheorem{lemma}{Lemma}
\newtheorem{proposition}{Proposition}
\newtheorem{corollary}{Corollary}
\newtheorem{definition}{Definition}

\newtheorem{example}{Example}

\DeclareMathOperator{\ann}{Ann}

\begin{document}

\title{The Coproduct of Ideals and the Coprime Spectrum}

\author{Frank Murphy-Hernandez}
\address{Facultad de Ciencias, UNAM, Mexico City}
\email{murphy@ciencias.unam.mx}

\author{Eduardo Leon-Rodriguez}
\address{Instituto de Matem\'aticas, UNAM, Mexico City}
\email{eduardoleon\_r@ciencias.unam.mx}


\date{\today}


\begin{abstract}
We introduce the coproduct of ideals and the notion of coprime ideal, extending the Heyting-algebra perspective on ideal lattices to arbitrary commutative rings. The resulting coprime spectrum is a topological space, defines a covariant functor on CF-morphisms of commutative rings, i.e. morphisms that extend ideals and are coprime faithful, classifies fields among integral domains, and in dual rings is homeomorphic to the prime spectrum. This offers an elementary, lattice-theoretic counterpart to the geometric study of nilpotents.
\end{abstract}

\maketitle

\section{Introduction}

In a commutative ring $A$, the set of ideals $\mathcal{I}(A)$ carries a rich lattice structure under inclusion, with intersection as meet and sum as join. When $A$ is Artinian semisimple, $\mathcal{I}(A)$ is in fact a Boolean algebra: the product and the intersection of ideals coincide, and the complement of an ideal $I$ is given by its annihilator $\ann(I)$. In the general case this perfect picture breaks down. The annihilator $\ann(I)$ is no longer a complement but a pseudocomplement: it is the largest ideal $K$ such that $IK=0$. The product, moreover, no longer coincides with the intersection, so it cannot serve as the meet of the lattice, and the ideal lattice of a general ring is far from Boolean. Nevertheless, the analogy with Boolean algebras suggests a fruitful line of inquiry, one that replaces the product by an operation better suited to the annihilator.

The starting point is the ideal quotient $(I:J)=\{a\in A\mid aJ\subseteq I\}$, which is the largest ideal $K$ such that $KJ\subseteq I$. In the language of Heyting algebras, $(I:J)$ plays the role of the implication $J\to I$. In a Boolean algebra $B$ the join is recovered from the pseudocomplement and the implication by the identity
\[
x\vee y = x^{*}\to y \qquad (x,y\in B),
\]
where $x^{*}$ denotes the complement of $x$. It is therefore natural to ask whether the sum of ideals can be recovered from the annihilator and the ideal quotient. In an Artinian semisimple ring this is indeed the case: for all ideals $I,J\subseteq A$,
\[
I+J = (J:\ann(I)).
\]
This observation motivates the central definition of the present work: the \emph{coproduct} of ideals
\[
[I:J] := (J:\ann(I)) = \{\,a\in A\mid a\cdot\ann(I)\subseteq J\,\}.
\]
The coproduct is thus a fourth operation on ideals, alongside the product, the intersection, and the sum, and it is tailored to interact with the annihilator. In an Artinian semisimple ring the sum and the coproduct coincide, so the coproduct may be regarded as a generalisation of the sum to arbitrary rings, one that remains well behaved in the presence of nilpotents and non-trivial annihilators. In this sense the present work extends the classical Heyting-algebra perspective on ideal lattices, where the ideal quotient provides the implication, to the ideal lattice of an arbitrary commutative ring, by adding a dual operation that plays the role of the disjunction.

Classical commutative algebra has traditionally favoured reduced rings, where nilpotent elements are absent and annihilators are often trivial. Yet nilpotents and, more generally, non-zero annihilators carry genuine algebraic information, information that Grothendieck's scheme theory incorporates through structure sheaves. The present article takes a more elementary, though philosophically aligned, approach: instead of working in the category of schemes, we develop a purely commutative-algebraic framework that exploits non-zero annihilators and the coproduct operation. This allows us to study phenomena that are invisible in the reduced setting, remaining entirely within the language of ideals, quotients, and annihilators, without invoking sheaf theory. The theory developed here may therefore be regarded as an alternative and parallel route to the geometric insights of scheme theory, one that privileges the lattice of ideals over the structure sheaf, and the annihilator over the nilpotent thickening.

The results below are stated for arbitrary commutative rings, but the theory is particularly well behaved in two regimes: dual rings, where it recovers the prime spectrum, and rings with non-trivial annihilators, where it produces finite spaces such as the Sierpi\'nski space.

The main contributions of this work are the following. We introduce the coproduct $[I:J]=(J:\ann(I))$ and establish its elementary theory, including its comparison with the classical operations on ideals and the chain of inclusions
\[
IJ\subseteq I\cap J\subseteq I+J\subseteq [I:J].
\]
We highlight the duality between the product and the coproduct: for every ideal $I$,
\[
I\ann(I)=0 \qquad\text{and}\qquad [I:\ann(I)]=A,
\]
so that the annihilator behaves as a complement with respect to both operations. Building on this, we define an ideal $C\neq 0$ to be \emph{coprime} if $C\subseteq [I:J]$ implies $C\subseteq I$ or $C\subseteq J$, mirroring the definition of prime ideals with respect to the product. We give an operational characterisation of coprimality, prove that it is stable under sums, and show that every idempotent minimal ideal is coprime, while minimal nilpotent ideals need not be. This leads us to construct the coprime spectrum $\operatorname{CoSp}(A)$ as the set of all coprime ideals, endowed with a topology whose closed sets are
\[
V_c(I)=\{\,C\in\operatorname{CoSp}(A)\mid C\subseteq I\,\}.
\]
We prove that this indeed defines a topological space, and we show that the assignment $A\mapsto \operatorname{CoSp}(A)$ is functorial with respect to CF-morphisms, that is, morphisms of commutative rings that extend ideals and are coprime faithful. Every such morphism $f\colon A\to B$ induces a continuous map $\operatorname{CoSp}(f)\colon \operatorname{CoSp}(A)\to \operatorname{CoSp}(B)$.

The coprime spectrum turns out to classify familiar algebraic objects. In the special case of integral domains we completely classify coprime ideals: such an ideal, if it exists, is the least non-zero ideal, and its existence forces the domain to be a field. In particular, for an integral domain $D$ we have
\[
|\operatorname{CoSp}(D)|\le 1,
\]
with equality if and only if $D$ is a field. Thus the coprime spectrum detects fields among integral domains, and more generally it provides a topological invariant that distinguishes reduced from non-reduced behaviour. For rings satisfying the double annihilator condition $\ann(\ann(I))=I$ --- the \emph{dual rings}, which include all commutative quasi-Frobenius rings --- we prove that an ideal $C$ is coprime if and only if its annihilator $\ann(C)$ is prime. Consequently the annihilator map induces a bijection
\[
\ann\colon \operatorname{CoSp}(A)\longrightarrow \operatorname{Spec}(A),
\]
which we show to be a homeomorphism when $\operatorname{CoSp}(A)$ is endowed with the topology above and $\operatorname{Spec}(A)$ with the Zariski topology. This places coprimality and primality in perfect correspondence and shows that, in the dual setting, the coprime spectrum carries exactly the same topological information as the prime spectrum.

Finally, we illustrate the theory with explicit computations in rings with zero divisors. In particular, we compute the coprime spectrum of the trivial extension $\mathbb{Z}\ltimes \mathbb{Z}_{p^\infty}$ of the integers by the Pr\"ufer $p$-group, and we show that it is homeomorphic to the Sierpi\'nski space. This example, together with the pathological behaviour exhibited by the ring $k[x,y]/(x^2,y^2,xy)$, demonstrates that the coproduct captures meaningful algebraic information that is invisible in the reduced setting, and reinforces the view that the coprime spectrum is a natural companion to the prime spectrum in the study of rings with non-trivial annihilators.

The structure of the paper is as follows. Section~\ref{sec:prelim} recalls the standard identities for ideal quotients and annihilators, together with the basic facts about trivial extensions and the Pr\"ufer $p$-group that will be needed later. Section~\ref{sec:coproduct} introduces the coproduct of ideals and establishes its elementary properties, including its relation to the sum of ideals and its partial associativity. Section~\ref{sec:coprime-ideal} defines coprime ideals, provides their main characterisations, and exhibits examples of idempotent minimal ideals that are coprime as well as nilpotent minimal ideals that are not. Section~\ref{sec:coprime-spectrum} constructs the coprime spectrum as a topological space, studies its behaviour under CF-morphisms of commutative rings, and shows that it defines a covariant functor. Section~\ref{sec:dual-rings} focuses on dual rings, where the coprime spectrum and the prime spectrum are shown to be homeomorphic via the annihilator map.

\section{Preliminaries}
\label{sec:prelim}

In this paper all rings are unitary commutative rings. For a ring $A$ and ideals $I,J\subseteq A$, the \emph{ideal quotient} $(I:J)$ is
\[
(I:J)=\{a\in A\mid aJ\subseteq I\}.
\]
The \emph{annihilator} of $I$ is
\[
\ann(I)=\{a\in A\mid aI=0\}.
\]
Standard references are \cite{atiyah2018introduction} and \cite{greuel2008singular}. We recall some standard identities concerning the ideal quotient and the annihilator, which will be useful in the sequel.

\begin{proposition}\label{pro:quotient}
Let $A$ be a ring, $J,K,L\subseteq A$, and $\{J_i\}_{i\in\Lambda}$ a family of ideals of $A$. Then:
\begin{enumerate}
  \item $J\subseteq (J:K)$;
  \item $(J:K)=A$ if and only if $K\subseteq J$;
  \item $\bigl((J:K):L\bigr)=(J:KL)$;
  \item $\displaystyle\Bigl(\bigcap_{i\in\Lambda} J_i:K\Bigr)=\bigcap_{i\in\Lambda}(J_i:K)$;
  \item $K\,(J:K)\subseteq J$.
\end{enumerate}
\end{proposition}

\begin{proposition}\label{pro:annihilator}
Let $A$ be a ring, $I,J\subseteq A$, and $\{I_i\}_{i\in\Lambda}$ a family of ideals of $A$. Then:
\begin{enumerate}
  \item $\ann(I)=(0:I)$;
  \item if $I\subseteq J$, then $\ann(J)\subseteq\ann(I)$;
  \item $I\subseteq\ann(\ann(I))$ and $\ann(I)=\ann(\ann(\ann(I)))$;
  \item $\displaystyle\ann\Bigl(\sum_{i\in\Lambda} I_i\Bigr)=\bigcap_{i\in\Lambda}\ann(I_i)$;
  \item $\ann(IJ)=(\ann(I):J)=(\ann(J):I)$.
\end{enumerate}
\end{proposition}

We will also need some basic facts about trivial extensions and the Pr\"ufer $p$-group.

\begin{definition}[Trivial extension]\label{def:trivial-extension}
Let $A$ be a commutative ring and $M$ an $A$-module. The \emph{trivial extension} of $A$ by $M$, denoted by $A\ltimes M$, is the ring whose underlying set is $A\oplus M$, with addition defined componentwise and multiplication given by
\[
(a,m)(b,n)=(ab,\,an+bm).
\]
This ring is commutative and has identity $(1,0)$.
\end{definition}

The following notion, standard in the theory of trivial extensions and in the study of commutative rings with zero divisors \cite{anderson2009idealization,huckaba1988commutative}, will allow us to give a correct description of the ideals of $A\ltimes M$.

\begin{definition}[Homogeneous ideal]\label{def:homogeneous-ideal}
Let $A$ be a commutative ring, $M$ an $A$-module, and $B=A\ltimes M$. An ideal $J\subseteq B$ is called \emph{homogeneous} if $(a,m)\in J$ then $(a,0)\in J$.
\end{definition}

\begin{definition}[Pr\"ufer $p$-group]\label{def:prufer}
Let $p$ be a prime. The \emph{Pr\"ufer $p$-group}, denoted by $\mathbb Z_{p^\infty}$, is the direct limit of the system
\[
\mathbb Z/p\mathbb Z \xrightarrow{\cdot p} \mathbb Z/p^2\mathbb Z \xrightarrow{\cdot p} \mathbb Z/p^3\mathbb Z \xrightarrow{\cdot p}\cdots
\]
Equivalently, $\mathbb Z_{p^\infty}$ is the subgroup of $\mathbb Q/\mathbb Z$ consisting of all elements whose order is a power of $p$. For $n\ge 1$, we denote by $C_{p^n}$ the unique subgroup of $\mathbb Z_{p^\infty}$ of order $p^n$.
\end{definition}

\begin{proposition}[Subgroups of the Pr\"ufer group]\label{prop:prufer-subgroups}
Let $p$ be a prime and $M=\mathbb Z_{p^\infty}$. Then:
\begin{enumerate}
  \item The proper non-zero subgroups of $M$ are precisely the cyclic groups $C_{p^n}$ for $n\ge 1$, and they form a chain
  \[
  C_p\subset C_{p^2}\subset C_{p^3}\subset\cdots\subset M.
  \]
  \item $M$ is divisible: for every $m\in M$ and every non-zero integer $d$, there exists $m'\in M$ such that $d m'=m$.
\end{enumerate}
\end{proposition}

\begin{proposition}[Homogeneous ideals of a trivial extension]
\label{prop:ideals-trivial-extension}
Let $A$ be a commutative ring, $M$ an $A$-module, and $B=A\ltimes M$. Then the homogeneous ideals of $B$ are in bijection with the pairs $(I,N)$, where $I$ is an ideal of $A$, $N$ is a submodule of $M$,
and $IM\subseteq N$. The bijection sends $(I,N)$ to the ideal
\[
I\oplus N=\{(a,m)\mid a\in I,\ m\in N\}.
\]
\end{proposition}

\begin{proof}
See \cite[Section 3]{anderson2009idealization}.
\end{proof}

\section{Coproduct of ideals}
\label{sec:coproduct}

The usual ideal quotient $(I:J)$ is an asymmetric operation that measures how much of $J$ is needed to multiply into $I$ to obtain an element of $I$. When working with annihilators, however, it is often convenient to have a dual construction that instead measures how much of the annihilator of $I$ is contained in $J$. This motivates the following definition.

\begin{definition}\label{def:coproduct}
Let $A$ be a ring and $I,J\subseteq A$. We define the \emph{coproduct} of $I$ and $J$, denoted by $[I:J]$, as the ideal quotient
\[
[I:J] := (J:\ann(I)).
\]
Equivalently,
\[
[I:J] = \{\, a \in A \mid a\cdot \ann(I) \subseteq J \,\}.
\]
Since ideal quotients are ideals, we have $[I:J]\subseteq A$.
\end{definition}

The following proposition records the elementary behaviour of the coproduct when one of the arguments is the zero ideal or the whole ring.

\begin{proposition}\label{prop:coprod-basic}
Let $A$ be a ring and $I,J\subseteq A$. Then:
\begin{enumerate}
  \item $[0:J]=J$;
  \item $[I:A]=A$;
  \item $[A:J]=A$;
  \item $[I:0]=\ann(\ann(I))$.
\end{enumerate}
\end{proposition}

\begin{proof}
By definition, $[I:J]=(J:\ann(I))$. Since $\ann(0)=A$, we obtain
\[
[0:J]=(J:A)=J,
\]
which proves (1). For (2), observe that $A\cdot \ann(I)\subseteq A$ for every ideal $I$, hence every element of $A$ belongs to $(A:\ann(I))$; therefore $(A:\ann(I))=A$, so $[I:A]=A$. For (3), since $\ann(A)=0$, we have
\[
[A:J]=(J:\ann(A))=(J:0)=A.
\]
Finally, for (4), again by definition,
\[
[I:0]=(0:\ann(I))=\ann(\ann(I)).
\]
This completes the proof.
\end{proof}

\begin{example}\label{ex:coproduct-Z}
In $\mathbb{Z}$, every ideal is of the form $m\mathbb{Z}$. Since $\mathbb{Z}$ is an integral domain,
\[
\ann(m\mathbb{Z})=
\begin{cases}
\mathbb{Z}, & m=0,\\
0, & m\neq 0.
\end{cases}
\]
Therefore:
\[
[m\mathbb{Z}:n\mathbb{Z}]=
\begin{cases}
n\mathbb{Z}, & m=0,\\[4pt]
\mathbb{Z}, & m\neq 0.
\end{cases}
\]
\end{example}

The coproduct can be seen as a right-adjoint operation to the ideal quotient in the following sense: while $(I:J)$ is the largest ideal $K$ such that $KJ\subseteq I$, the coproduct $[I:J]$ is the largest ideal $K$ such that $K\cdot \ann(I)\subseteq J$.

Next, we investigate the associativity-like behaviour of the coproduct.

\begin{proposition}\label{prop:coproduct-associative}
Let $A$ be a ring and $I,J,K\subseteq A$. Then
\[
[[I:J]:K] \subseteq [I:[J:K]].
\]
\end{proposition}

\begin{proof}
Let $x\in [[I:J]:K]$ and $a\in \ann(I)$. It suffices to show that $ax\in (K:\ann(J))$. Let $b\in \ann(J)$. We must verify that $(xa)b\in K$. Since $x\in [[I:J]:K]$, it is enough to prove that $ab\in \ann([I:J])$. Let $c\in [I:J]$. Then
\[
(ab)c = a(bc).
\]
Because $c\in [I:J]=(J:\ann(I))$ and $a\in \ann(I)$, we have $ac\in J$. Since $b\in \ann(J)$, it follows that $b(ac)=0$. Hence $(ab)c=0$ for all $c\in [I:J]$, so $ab\in \ann([I:J])$. Therefore $x(ab)\in K$, and consequently $ax\in (K:\ann(J))=[J:K]$. Thus $x\in ([J:K]:\ann(I))=[I:[J:K]]$, as required.
\end{proof}

The reverse inclusion does not hold in general, as the following example shows.

\begin{example}\label{ex:nonassoc}
The reverse inclusion
\[
[I:[J:K]] \subseteq [[I:J]:K]
\]
does not hold in general. To see this, let $k$ be a field and consider the ring
\[
A = k[x,y]/\langle x^2, y^2, xy\rangle,
\]
where $\bar{x}$ and $\bar{y}$ denote the classes of $x$ and $y$ in the quotient. This is a local ring with maximal ideal $\mathfrak{m}=\langle \bar{x}, \bar{y} \rangle$, and one readily checks that $\mathfrak{m}^2=0$ and $\ann(\mathfrak{m})=\mathfrak{m}$. Moreover, $\ann(\langle \bar{x} \rangle)=\mathfrak{m}$.

Take the ideals
\[
I=\langle \bar{x} \rangle,\qquad J=\langle \bar{x} \rangle,\qquad K=\langle \bar{x} \rangle.
\]
Then $\ann(I)=\mathfrak{m}$. Hence
\[
[I:J]=(J:\ann(I))=(\langle \bar{x} \rangle:\mathfrak{m}).
\]
Now $a\cdot \mathfrak{m}\subseteq \langle \bar{x} \rangle$ if and only if $a\in \mathfrak{m}$, so $(\langle \bar{x} \rangle:\mathfrak{m})=\mathfrak{m}$. Thus $[I:J]=\mathfrak{m}$, and since $\ann(\mathfrak{m})=\mathfrak{m}$, we get
\[
[[I:J]:K]=[\mathfrak{m}:\langle \bar{x} \rangle]=(\langle \bar{x} \rangle:\ann(\mathfrak{m}))=(\langle \bar{x} \rangle:\mathfrak{m})=\mathfrak{m}.
\]
On the other hand,
\[
[J:K]=[\langle \bar{x} \rangle:\langle \bar{x} \rangle]=(\langle \bar{x} \rangle:\ann(\langle \bar{x} \rangle))=(\langle \bar{x} \rangle:\mathfrak{m})=\mathfrak{m},
\]
and therefore
\[
[I:[J:K]]=[\langle \bar{x} \rangle:\mathfrak{m}]=(\mathfrak{m}:\ann(\langle \bar{x} \rangle))=(\mathfrak{m}:\mathfrak{m})=A.
\]
Thus
\[
[[I:J]:K]=\mathfrak{m} \quad\text{and}\quad [I:[J:K]]=A.
\]
Since $1\notin \mathfrak{m}$, we have $[I:[J:K]]\nsubseteq [[I:J]:K]$. Consequently, the inclusion in the opposite direction fails in general.
\end{example}

We continue by establishing a fundamental inclusion that relates the coproduct to the sum of ideals.

\begin{proposition}\label{pro:sum}
Let $A$ be a ring and $I,J\subseteq A$. Then
\[
I+J \subseteq [I:J].
\]
\end{proposition}

\begin{proof}
Let $x\in I+J$ and $y\in \ann(I)$. By definition, $x=a+b$ for some $a\in I$ and $b\in J$. Since $a\in I$ and $y\in \ann(I)$, we have $ay=0$. Hence
\[
xy=(a+b)y=ay+by=by\in J,
\]
because $b\in J$ and $J$ is an ideal. Thus every element of $I+J$ satisfies the condition to belong to $(J:\ann(I))=[I:J]$. Therefore $I+J\subseteq [I:J]$.
\end{proof}

The inclusion in Proposition~\ref{pro:sum} is generally strict. For instance, in the ring
\[
A=k[x,y]/\langle x^2,y^2,xy\rangle,
\]
let $\mathfrak m=\langle \bar{x},\bar{y}\rangle$ be the maximal ideal, and take
\[
I=\langle \bar{x}\rangle,\qquad J=0.
\]
Then $I+J=\langle \bar{x}\rangle$, while
\[
[I:J]=[I:0]=\ann(\ann(I))=\ann(\mathfrak m)=\mathfrak m.
\]
Since $\langle \bar{x}\rangle\subsetneq \mathfrak m$, the inclusion $I+J\subseteq [I:J]$ is proper.

The coproduct provides a fourth operation on ideals, alongside the product, intersection, and sum. The following result establishes a fundamental chain of inclusions that compares all four operations.

\begin{corollary}\label{cor:order}
Let $A$ be a ring and $I,J\subseteq A$. Then
\[
IJ \subseteq I\cap J \subseteq I+J \subseteq [I:J].
\]
\end{corollary}

The following proposition establishes a crucial link between the annihilator and the coproduct, revealing a complementary behaviour with respect to the unit ideal.

\begin{proposition}\label{prop:coproduct-annihilator}
Let $A$ be a commutative ring and $I$ an ideal of $A$. Then
\[
[I:\ann(I)] = A,
\]
and if $J \subseteq A$ satisfies $[I:J] = A$, then $\ann(I) \subseteq J$.
\end{proposition}

\begin{proof}
By definition, $[I:\ann(I)] = (\ann(I):\ann(I))$. Since $\ann(I) \subseteq \ann(I)$, it follows from Proposition~\ref{pro:quotient}.(2) that $(\ann(I):\ann(I)) = A$, proving the first assertion.

Now suppose $[I:J] = A$. Again by definition, $[I:J] = (J:\ann(I)) = A$. By Proposition~\ref{pro:quotient}.(2), this is equivalent to $\ann(I) \subseteq J$, as required.
\end{proof}

The previous proposition unveils a fundamental duality between the product and the coproduct of ideals. For a fixed ideal $I$, the annihilator $\ann(I)$ is the largest ideal $K$ such that $IK=0$: indeed, $I\cdot\ann(I)=0$, and if $K$ is any ideal with $IK=0$, then by definition $K\subseteq\ann(I)$. On the other hand, $\ann(I)$ is the smallest ideal $J$ such that $[I:J]=A$: we have $[I:\ann(I)]=A$, and if $[I:J]=A$, then Proposition~\ref{prop:coproduct-annihilator} gives $\ann(I)\subseteq J$. Thus, in the lattice of ideals, the pair $(I,\ann(I))$ behaves as a kind of complement with respect to the two operations:
\[
I\ann(I) = 0 \qquad \text{and} \qquad [I:\ann(I)] = A.
\]
The product yields the zero ideal, while the coproduct yields the unit ideal. This symmetric duality justifies the term ``coproduct'' for the operation $[\cdot,\cdot]$, since it is dual to the usual product in the sense that the annihilator serves as a complement for both operations. In the special case of an Artinian semisimple ring, where the product and intersection coincide and $\ann(I)$ is the Boolean complement of $I$, the coproduct reduces to the usual sum, recovering the classical complementation of the ideal lattice \cite{ChajdaLanger2019}.

\section{Coprime Ideals}
\label{sec:coprime-ideal}

In classical commutative algebra, a prime ideal $P$ is defined by the property that $IJ \subseteq P$ implies $I \subseteq P$ or $J \subseteq P$. This captures the notion of irreducibility with respect to the product of ideals. Since the coproduct $[\cdot,\cdot]$ behaves dually to the product --- replacing $IJ=0$ by $[I:J]=A$ --- it is natural to define a corresponding notion of primality with respect to this new operation. We call such ideals \emph{coprime}.

\begin{definition}\label{def:coprime}
Let $A$ be a ring and $C\subseteq A$. We say that $C$ is a \emph{coprime ideal} if $C\neq 0$ and whenever $I,J\subseteq A$ satisfy
\[
C \subseteq [I:J],
\]
then
\[
C \subseteq I \quad \text{or} \quad C \subseteq J.
\]
\end{definition}

Equivalently, an ideal $C\neq 0$ is coprime if and only if, for all $I,J\subseteq A$,
\[
C\cdot \ann(I) \subseteq J \quad \Longrightarrow \quad C\subseteq I \ \text{or} \ C\subseteq J,
\]
since $C\subseteq [I:J]=(J:\ann(I))$ is precisely the condition $C\cdot\ann(I)\subseteq J$. This formulation highlights the duality with prime ideals, where $IJ\subseteq P$ implies containment.

The following characterization will be used repeatedly.

\begin{lemma}\label{lem:coprime-characterization}
Let $A$ be a ring and $C\subseteq A$ with $C\neq 0$. Then $C$ is coprime if and only if $C\ann(I)=C$ for every ideal $I\subseteq A$ with $C\not\subseteq I$.
\end{lemma}

\begin{proof}
Suppose that $C$ is coprime and let $I\subseteq A$ with $C\not\subseteq I$. Put $J:=C\ann(I)$. Then $C\ann(I)\subseteq J$, that is, $C\subseteq(J:\ann(I))=[I:J]$. Since $C$ is coprime and $C\not\subseteq I$, we have that $C\subseteq J=C\ann(I)$, and the other containment is clear.

Conversely, suppose that $C\ann(I)=C$ whenever $C\not\subseteq I$, and let $I,J\subseteq A$ with $C\subseteq[I:J]=(J:\ann(I))$, that is, $C\ann(I)\subseteq J$. If $C\subseteq I$ we are done. Otherwise $C=C\ann(I)\subseteq J$. Therefore $C$ is coprime.
\end{proof}

Just as, for a prime ideal, the defining property with respect to the product passes to the intersection, so for a coprime ideal, the defining property with respect to the coproduct passes to the sum.

\begin{proposition}\label{prop:coprime-sum}
Let $C$ be a coprime ideal. If $I,J\subseteq A$ satisfy
\[
C \subseteq I+J,
\]
then
\[
C \subseteq I \quad \text{or} \quad C \subseteq J.
\]
\end{proposition}

\begin{proof}
By Proposition~\ref{pro:sum}, we have $I+J \subseteq [I:J]$. Hence $C\subseteq I+J$ implies $C\subseteq [I:J]$. Since $C$ is coprime, it follows that $C\subseteq I$ or $C\subseteq J$.
\end{proof}

Recall the classical Brauer's lemma: for a minimal ideal $I$, either $I^2=0$ or $I=Ae$ with $e\in A$ an idempotent element \cite{lam1991first}. This yields a natural supply of coprime ideals.

\begin{proposition}\label{prop:minimal-idempotent-coprime}
Let $A$ be a ring. If $M\subseteq A$ is an idempotent minimal ideal, then $M$ is coprime.
\end{proposition}

\begin{proof}
Let $I,J\subseteq A$ be ideals such that $M\subseteq [I:J]$. If $M\subseteq I$, there is nothing to prove. Suppose instead that $M\not\subseteq I$. Since $M$ is minimal, the ideal $I\cap M$ is properly contained in $M$, hence $I\cap M=0$. In particular,
\[
IM \subseteq I\cap M = 0,
\]
which implies $M\subseteq \ann(I)$.

On the other hand, from the definition of the coproduct, we have $[I:J]\,\ann(I)\subseteq J$. Consequently, using the idempotence of $M$ (i.e., $M^2=M$), we obtain
\[
M = M^2 \subseteq [I:J]\,\ann(I) \subseteq J.
\]
Therefore, in all cases, either $M\subseteq I$ or $M\subseteq J$. Hence $M$ is coprime.
\end{proof}

Proposition~\ref{prop:minimal-idempotent-coprime} shows that the coprime spectrum is non-empty whenever the ring admits a non-zero idempotent minimal ideal. In particular, in a semisimple Artinian ring, every simple ideal (which is idempotent) is coprime. This recovers the classical observation that the minimal ideals of a semisimple ring behave like ``atoms'' with respect to the lattice of ideals. Every minimal ideal of a ring is either idempotent or nilpotent. Indeed, if $M$ is a minimal ideal, then either $M^2=M$ or $M^2=0$. In the idempotent case, Proposition~\ref{prop:minimal-idempotent-coprime} guarantees that $M$ is coprime. The nilpotent case, however, does not always yield a coprime ideal, as the following example shows.

\begin{example}\label{ex:nilpotent-minimal-not-coprime}
Let $k$ be a field and consider the commutative ring
\[
A = k[x,y]/\langle x^2, xy, y^2\rangle,
\]
where $\bar{x}$ and $\bar{y}$ denote the classes of $x$ and $y$ in the quotient.  Let $\mathfrak{m} = \langle \bar{x}, \bar{y} \rangle $ be the unique maximal ideal of $A$. Note that $\mathfrak{m}^2 = 0$. We now show that the minimal ideal $C = \langle \bar{x} \rangle$ fails to be coprime.

Choose $I = J = \langle \bar{y} \rangle$. We first compute $\ann(I)$. For an element $u = a + b\bar{x} + c\bar{y} \in A$, we have
\[
u\bar{y} = (a + b\bar{x} + c\bar{y})\bar{y} = a\bar{y} + b\bar{x}\bar{y} + c\bar{y}^2 = a\bar{y}.
\]
Thus $u \in \ann(\langle \bar{y} \rangle)$ if and only if $a\bar{y} = 0$, which implies $a = 0$. Hence
\[
\ann(\langle \bar{y} \rangle) = \{\, b\bar{x} + c\bar{y} \mid b,c \in k \,\} = \mathfrak{m}.
\]

Next, we compute
\[
[I:J] = [\langle \bar{y} \rangle:\langle \bar{y} \rangle] = (\langle \bar{y} \rangle : \ann(\langle \bar{y} \rangle)) = (\langle \bar{y} \rangle : \mathfrak{m}).
\]
By definition,
\[
(\langle \bar{y} \rangle : \mathfrak{m}) = \{\, u \in A \mid u\mathfrak{m} \subseteq \langle \bar{y} \rangle \,\}.
\]
Take $u = a + b\bar{x} + c\bar{y} \in A$ and an arbitrary element $m = d\bar{x} + e\bar{y} \in \mathfrak{m}$, where $d,e \in k$. Since $\mathfrak{m}^2 = 0$, we have
\[
u m = (a + b\bar{x} + c\bar{y})(d\bar{x} + e\bar{y}) = ad\bar{x} + ae\bar{y}.
\]
For this product to lie in $\langle \bar{y} \rangle$ for every $d,e \in k$, the coefficient of $\bar{x}$ must vanish, so $ad = 0$ for all $d \in k$. Hence $a = 0$. Therefore $u \in \mathfrak{m}$. Conversely, if $u \in \mathfrak{m}$, then $u\mathfrak{m} \subseteq \mathfrak{m}^2 = 0 \subseteq \langle \bar{y} \rangle$. Thus
\[
(\langle \bar{y} \rangle : \mathfrak{m}) = \mathfrak{m}.
\]
So we conclude that
\[
[\langle \bar{y} \rangle:\langle \bar{y} \rangle] = \mathfrak{m}.
\]

Now observe that
\[
C = \langle \bar{x} \rangle \subseteq \mathfrak{m} = [\langle \bar{y} \rangle:\langle \bar{y} \rangle].
\]
However,
\[
C = \langle \bar{x} \rangle \not\subseteq I = \langle \bar{y} \rangle
\]
and
\[
C = \langle \bar{x} \rangle \not\subseteq J = \langle \bar{y} \rangle,
\]
because $\bar{x} \notin \langle \bar{y} \rangle$. Therefore, $C$ satisfies $C \subseteq [I:J]$, but it is contained in neither $I$ nor $J$. So $C$ is \textbf{not} a coprime ideal.

Since $C = \langle \bar{x} \rangle$ is a minimal ideal of $A$ (it is a one-dimensional $k$-subspace of the maximal ideal $\mathfrak{m}$, and $\mathfrak{m}^2 = 0$), this example proves that minimal nilpotent ideals are not necessarily coprime.
\end{example}

\section{Coprime Spectrum}
\label{sec:coprime-spectrum}

In this section we study the set of coprime ideals of a ring, which we call the coprime spectrum. We also investigate the special case of integral domains, where coprime ideals admit a particularly simple characterisation.

\begin{definition}\label{def:cosp}
Let $A$ be a ring. We define the \emph{coprime spectrum} of $A$, denoted by $\operatorname{CoSp}(A)$, as the set of all coprime ideals of $A$:
\[
\operatorname{CoSp}(A) := \{\, C \subseteq A \mid C \text{ is a coprime ideal} \,\}.
\]
\end{definition}

\begin{definition}\label{def:least-ideal}
Let $A$ be a ring. A non-zero ideal $M \subseteq A$ is called the \emph{least ideal} of $A$ if $M \subseteq I$ for every non-zero ideal $I \subseteq A$.
\end{definition}

The following result shows that in an integral domain, the only possible coprime ideals are the least ideal (when it exists), and conversely the least ideal is always coprime.

\begin{proposition}\label{prop:domain-coprime-least}
Let $D$ be an integral domain and $C \subseteq D$. Then $C$ is a coprime ideal if and only if $C$ is the least ideal of $D$.
\end{proposition}

\begin{proof}
$(\Rightarrow)$ Assume $C$ is coprime and let $I \subseteq D$ be a non-zero ideal. Since $D$ is a domain, $\ann(I)=0$. Therefore,
\[
[I:I] = (I:\ann(I)) = (I:0) = D,
\]
because $(I:0) = D$ (indeed, every element of $D$ annihilates $0$, so the quotient is the whole ring). Hence $C \subseteq D = [I:I]$. Since $C$ is coprime and both arguments of the coproduct are $I$, we have $C \subseteq I$. As $I$ was arbitrary non-zero, $C$ is contained in every non-zero ideal, so $C$ is the least ideal of $D$.

$(\Leftarrow)$ Assume now that $C$ is the least ideal of $D$. Let $I,J \subseteq D$ such that $C \subseteq [I:J]$. In particular, $[I:J] \neq 0$ because $C \neq 0$. We claim that this implies $I \neq 0$ or $J \neq 0$. Indeed, if $I=J=0$, then
\[
[0:0] = (0:\ann(0)) = (0:D) = 0,
\]
because in a domain $(0:D)=0$, contradicting $[0:0] \neq 0$. Thus at least one of $I,J$ is non-zero. Since $C$ is the least ideal, if $I \neq 0$, then $C \subseteq I$; if instead $J \neq 0$, then $C \subseteq J$. Therefore, in all cases, either $C \subseteq I$ or $C \subseteq J$. Hence $C$ is coprime.
\end{proof}

Proposition~\ref{prop:domain-coprime-least} shows that in an integral domain, the coprime spectrum contains at most one element: the least non-zero ideal, if it exists. Moreover, if the domain has a least ideal, then that ideal is automatically coprime. This provides a complete classification of coprime ideals in integral domains.

\begin{corollary}\label{cor:domain-field}
Let $D$ be an integral domain. If $\operatorname{CoSp}(D) \neq \varnothing$, then $D$ is a field.
\end{corollary}

\begin{proof}
Let $C \in \operatorname{CoSp}(D)$. By Proposition~\ref{prop:domain-coprime-least}, $C$ is the least ideal of $D$. In particular, $C$ is a minimal non-zero ideal. We now apply Brauer's lemma, which states that if $M$ is a minimal ideal of a ring and $M^2 \neq 0$, then $M = De$ for some idempotent element $e \in D$.

Since $D$ is an integral domain and $C \neq 0$, we have $C^2 \neq 0$ (otherwise, $C$ would contain a non-zero nilpotent element, contradicting the domain property). Hence, by Brauer's lemma, there exists an idempotent $e \in D$ such that
\[
C = De.
\]
In an integral domain, the only idempotents are $0$ and $1$. Since $C \neq 0$, we must have $e = 1$. Therefore,
\[
C = D.
\]
Now, because $C$ is the least ideal, every non-zero ideal $I \subseteq D$ satisfies $C = D \subseteq I$, hence $I = D$. Thus the only non-zero ideal of $D$ is $D$ itself. Consequently, for every non-zero element $a \in D$, the ideal $\langle a \rangle$ is non-zero, so $\langle a \rangle = D$. This means $a$ is a unit. Therefore, every non-zero element of $D$ is invertible, so $D$ is a field.
\end{proof}

Corollary~\ref{cor:domain-field} shows that among integral domains, fields are precisely those with non-empty coprime spectrum. Combined with Proposition~\ref{prop:domain-coprime-least}, this gives a complete characterisation: in an integral domain, either the coprime spectrum is empty, or it consists of a single element, namely the whole ring $D$, and $D$ is a field.

\begin{corollary}\label{cor:domain-cosp-cardinality}
Let $D$ be an integral domain. Then $|\operatorname{CoSp}(D)| \leq 1$.
\end{corollary}

\begin{proof}
By Proposition~\ref{prop:domain-coprime-least}, every coprime ideal of an integral domain is the least ideal of $D$. Since the least ideal, when it exists, is unique, the coprime spectrum contains at most one element. Hence $|\operatorname{CoSp}(D)| \leq 1$.
\end{proof}

We now characterise when the whole ring is a coprime ideal.

\begin{proposition}\label{prop:ring-field-coprime}
Let $A$ be a commutative ring. Then $A$ is a field if and only if $A$ itself is a coprime ideal.
\end{proposition}

\begin{proof}
$(\Rightarrow)$ Suppose $A$ is a field. Then the only ideals of $A$ are $0$ and $A$. Let $I,J \subseteq A$ be such that $A \subseteq [I:J]$. Since $[I:J]$ is an ideal and $A$ is the whole ring, this means $[I:J] = A$. We claim that $I = A$ or $J = A$. Suppose $I \neq A$, so $I = 0$. Then $\ann(I) = \ann(0) = A$, and hence
\[
[0:J] = (J:\ann(0)) = (J:A) = J,
\]
where the last equality holds because $J$ is an ideal. Thus $[0:J] = A$ implies $J = A$. Therefore, if $I \neq A$, then $J = A$. In all cases, $I = A$ or $J = A$, which means $A \subseteq I$ or $A \subseteq J$. Hence $A$ is coprime.

$(\Leftarrow)$ Conversely, suppose $A$ is a coprime ideal. We show that $A$ has no proper non-zero ideals. Let $I \subseteq A$ be a non-zero ideal. If $I$ is proper, set $J = \ann(I)$. By Proposition~\ref{prop:coproduct-annihilator}, we have $[I:J] = [I:\ann(I)] = A$. Since $A$ is coprime and $A \subseteq [I:J]$, we obtain
\[
A \subseteq I \quad \text{or} \quad A \subseteq J.
\]
The first inclusion would imply $I = A$, contradicting that $I$ is proper. Therefore $A \subseteq J = \ann(I)$, so $\ann(I) = A$. But then for every $a \in A$ and every $x \in I$, we have $ax = 0$; taking $a = 1$ yields $x = 0$ for all $x \in I$, so $I = 0$, contradicting that $I$ is non-zero. Thus no proper non-zero ideal exists. Hence every non-zero element is a unit, and $A$ is a field.
\end{proof}

Proposition~\ref{prop:ring-field-coprime} shows that the whole ring $A$ is a coprime ideal precisely when $A$ is a field. This should not be confused with Corollary~\ref{cor:domain-cosp-cardinality}, which states that an integral domain has at most one coprime ideal, and if it exists, it must be the whole ring (which forces $D$ to be a field by Corollary~\ref{cor:domain-field}). In other words, for integral domains, the existence of any coprime ideal is equivalent to the whole ring being coprime, and both are equivalent to the ring being a field.

We now introduce a topological structure on the coprime spectrum.

\begin{definition}\label{def:Vc}
Let $A$ be a ring and $I \subseteq A$. We define
\[
V_c(I) := \{\, C \in \operatorname{CoSp}(A) \mid C \subseteq I \,\}.
\]
\end{definition}

\begin{proposition}\label{prop:Vc-properties}
Let $A$ be a ring, $I, J \subseteq A$, and $\{I_\alpha\}_{\alpha \in \Gamma}$ a family of ideals of $A$. Then:
\begin{enumerate}
  \item $V_c(0) = \varnothing$;
  \item $V_c(A) = \operatorname{CoSp}(A)$;
  \item $V_c([I:J]) = V_c(I) \cup V_c(J)$;
  \item $\displaystyle V_c\!\left(\bigcap_{\alpha \in \Gamma} I_\alpha\right) = \bigcap_{\alpha \in \Gamma} V_c(I_\alpha)$.
\end{enumerate}
\end{proposition}

\begin{proof}
(1) By definition, $V_c(0)$ consists of coprime ideals $C$ such that $C \subseteq 0$. Since every coprime ideal is non-zero, no such $C$ exists. Hence $V_c(0) = \varnothing$.

(2) Since every ideal $C \subseteq A$ satisfies $C \subseteq A$, we have
\[
V_c(A) = \{ C \in \operatorname{CoSp}(A) \mid C \subseteq A \} = \operatorname{CoSp}(A).
\]

(3) We prove the two inclusions.

$(\subseteq)$ Let $C \in V_c([I:J])$. Then $C \subseteq [I:J]$. Since $C$ is a coprime ideal, by Definition~\ref{def:coprime} we have $C \subseteq I$ or $C \subseteq J$. Hence $C \in V_c(I) \cup V_c(J)$.

$(\supseteq)$ Conversely, suppose $C \in V_c(I) \cup V_c(J)$. Then either $C \subseteq I$ or $C \subseteq J$. Since by Proposition~\ref{pro:sum} we have $I+J \subseteq [I:J]$, it follows that $I \subseteq [I:J]$ and $J \subseteq [I:J]$. Hence $C \subseteq [I:J]$. Therefore $C \in V_c([I:J])$, and the equality holds.

(4) We again prove both inclusions.

$(\subseteq)$ Let $C \in V_c\!\left(\bigcap_{\alpha} I_\alpha\right)$. Then
\[
C \subseteq \bigcap_{\alpha \in \Gamma} I_\alpha.
\]
Thus, for each $\alpha \in \Gamma$, we have $C \subseteq I_\alpha$, so $C \in V_c(I_\alpha)$ for every $\alpha$. Therefore
\[
C \in \bigcap_{\alpha \in \Gamma} V_c(I_\alpha).
\]

$(\supseteq)$ Conversely, suppose
\[
C \in \bigcap_{\alpha \in \Gamma} V_c(I_\alpha).
\]
Then $C \subseteq I_\alpha$ for every $\alpha \in \Gamma$, which is equivalent to
\[
C \subseteq \bigcap_{\alpha \in \Gamma} I_\alpha.
\]
Hence $C \in V_c\!\left(\bigcap_{\alpha} I_\alpha\right)$. This proves (4).
\end{proof}

The properties established in Proposition~\ref{prop:Vc-properties} show that the family of sets
\[
\mathcal{F} := \{\, V_c(I) \mid I \subseteq A \,\}
\]
satisfies the axioms for closed sets of a topology on $\operatorname{CoSp}(A)$:
\begin{itemize}
  \item $\varnothing$ and $\operatorname{CoSp}(A)$ are in $\mathcal{F}$ (by (1) and (2));
  \item $\mathcal{F}$ is closed under finite unions (by (3), extended by induction);
  \item $\mathcal{F}$ is closed under arbitrary intersections (by (4)).
\end{itemize}
Therefore, we can view $\operatorname{CoSp}(A)$ as a topological space, called the \emph{coprime spectrum} of $A$, whose closed sets are precisely the $V_c(I)$ for ideals $I$ of $A$. This topology is analogous to the Zariski topology on the prime spectrum \cite{atiyah2018introduction}.

\begin{lemma}\label{lem:ideales-prufer-homogeneos}
Let $p$ be a prime, $M=\mathbb Z_{p^\infty}$, and $A=\mathbb Z\ltimes M$. Then every ideal of $A$ is homogeneous.
\end{lemma}

\begin{proof}
Let $J\subseteq A$ be an ideal. If every element of $J$ has zero $\mathbb Z$-component, then $J=0\oplus N$ for some subgroup $N\le M$, and $J$ is homogeneous.

Otherwise, there exists $(a,x)\in J$ with $a\neq 0$. For every $y\in M$,
\[
(0,y)(a,x)=(0,ay)\in J.
\]
Since $M$ is divisible, $aM=M$, so $0\oplus M\subseteq J$. Now, for any $(b,z)\in J$, we have $(0,z)\in J$, hence
\[
(b,0)=(b,z)-(0,z)\in J.
\]
Therefore $J$ is homogeneous.
\end{proof}

\begin{proposition}\label{prop:cosp-trivial-extension}
Let $p$ be a prime, let $M=\mathbb Z_{p^\infty}$ be the Pr\"ufer $p$-group, and let
\[
A=\mathbb Z\ltimes M
\]
be the trivial extension of $\mathbb Z$ by $M$, with multiplication
\[
(a,x)(b,y)=(ab,\,ay+bx).
\]
For $n\ge 1$, let $C_{p^n}$ denote the unique subgroup of $M$ of order $p^n$. Then
\[
\operatorname{CoSp}(A)=\{\,C_p,\;M\,\},
\]
where $C_p$ denotes the unique subgroup of $M$ of order $p$.
\end{proposition}

\begin{proof}
By Lemma~\ref{lem:ideales-prufer-homogeneos}, every ideal of $A$ is homogeneous. Hence Proposition~\ref{prop:ideals-trivial-extension} applies, and the ideals of $A$ correspond to pairs $(d\mathbb Z,N)$, where $d\ge 0$, $N$ is a subgroup of $M$, and $d\mathbb Z\cdot M\subseteq N$. If $d\neq 0$, then $dM=M$ because $M$ is divisible (Proposition~\ref{prop:prufer-subgroups}), so the condition $dM\subseteq N$ forces $N=M$. Hence the non-zero ideals of $A$, up to the natural identification of subgroups of $M$ with ideals $0\oplus N$, are
\[
C_{p^n}\ (n\ge 1),\qquad M,\qquad d\mathbb Z\oplus M\ (d>0).
\]

A direct computation using the multiplication rule gives the following annihilators:
\[
\begin{aligned}
\ann(0)&=A,\\
\ann(C_{p^n})&=p^n\mathbb Z\oplus M,\\
\ann(M)&=M,\\
\ann(d\mathbb Z\oplus M)&=0\oplus C_{p^{v_p(d)}}\quad (d>0),
\end{aligned}
\]
where $v_p(d)$ is the $p$-adic valuation of $d$.

We now determine which ideals are coprime. First, $M$ is coprime. Suppose $M\subseteq [I:J]$, that is, $M\cdot\ann(I)\subseteq J$. If $\ann(I)$ contains an element $(a,x)$ with $a\neq 0$, then $M\cdot\ann(I)=M$, because $aM=M$. Hence $M\subseteq J$. If every element of $\ann(I)$ has zero $\mathbb Z$-component, then $\ann(I)\subseteq M$. Looking at the list of ideals, this happens only for $I=d\mathbb Z\oplus M$ or $I=M$, and in both cases $M\subseteq I$. Therefore $M\subseteq I$ or $M\subseteq J$, so $M$ is coprime.

Second, $C_p$ is coprime. Let $C=C_p$ and suppose $C\subseteq [I:J]$, that is, $C\cdot\ann(I)\subseteq J$. Checking the list of ideals, we see that $C\cdot\ann(I)$ is either $0$ or $C$. If $C\cdot\ann(I)=C$, then $C\subseteq J$. If $C\cdot\ann(I)=0$, then $I\neq 0$ (otherwise $\ann(I)=A$ and $C\cdot\ann(I)=C$, a contradiction). In this case $I$ contains $C$: if $I=C_{p^n}$ with $n\ge 1$, then $C_p\subseteq C_{p^n}=I$; if
$I=d\mathbb Z\oplus M$ with $d>0$, then $M\subseteq I$, hence $C_p\subseteq I$; and if $I=M$, then $C_p\subseteq M=I$. Thus $C\subseteq I$. Therefore in all cases $C\subseteq I$ or $C\subseteq J$, so $C_p$ is coprime.

Finally, we show that no other non-zero ideal is coprime.  For $n\ge 2$, take $I=C_{p^{n-1}}$ and $J=C_p$. Then
  \[
  C_{p^n}\cdot\ann(I)=p^{n-1}C_{p^n}=C_p\subseteq J,
  \]
  so $C_{p^n}\subseteq [I:J]$. However,
  \[
  C_{p^n}\not\subseteq C_{p^{n-1}}=I
  \quad\text{and}\quad
  C_{p^n}\not\subseteq C_p=J.
  \]
  Hence $C_{p^n}$ is not coprime. For $d>0$, take $I=J=M$. Then $\ann(M)=M$, and
  \[
  [M:M]=(M:M)=A.
  \]
  Thus $d\mathbb Z\oplus M\subseteq A=[M:M]$, but
  \[
  d\mathbb Z\oplus M\not\subseteq M.
  \]
  Hence $d\mathbb Z\oplus M$ is not coprime. The zero ideal is not coprime by definition. Therefore
\[
\operatorname{CoSp}(A)=\{C_p,M\}.
\]
\end{proof}
\begin{corollary}\label{cor:cosp-trivial-extension-sierpinski}
In the situation of Proposition~\ref{prop:cosp-trivial-extension}, the coprime spectrum
\[
\operatorname{CoSp}(A)=\{C_p,M\}
\]
is homeomorphic to the Sierpi\'nski space. Indeed, its closed sets are
\[
\varnothing,\qquad V_c(C_p)=\{C_p\},\qquad V_c(M)=\operatorname{CoSp}(A),
\]
so its open sets are
\[
\varnothing,\qquad \{M\},\qquad \operatorname{CoSp}(A).
\]
Thus $M$ is the generic point and $C_p$ is the unique closed point.
\end{corollary}

\begin{proof}
By Proposition~\ref{prop:cosp-trivial-extension}, the only coprime ideals are $C_p$ and $M$, and $C_p\subseteq M$. Therefore
\[
V_c(C_p)=\{C\in\operatorname{CoSp}(A)\mid C\subseteq C_p\}=\{C_p\},
\]
while
\[
V_c(M)=\{C\in\operatorname{CoSp}(A)\mid C\subseteq M\}=\{C_p,M\}=\operatorname{CoSp}(A).
\]
Together with $V_c(0)=\varnothing$ and $V_c(A)=\operatorname{CoSp}(A)$, these are exactly the closed sets of the Sierpi\'nski topology on a two-point set. The corresponding open sets are their complements:
\[
\operatorname{CoSp}(A),\qquad \{M\},\qquad \varnothing.
\]
Hence the space is Sierpi\'nski. The point $M$ is generic and $C_p$ is the unique closed point.
\end{proof}

We now study how the coproduct behaves with respect to a convenient class of morphisms, namely those that extend ideals. This class contains all surjective morphisms and all localizations, and it is closed under composition and contains identities. We also need a condition preventing the annihilation of coprime ideals.

\begin{definition}\label{def:extends-ideals}
Let $f\colon A\to B$ be a morphism of commutative rings. We say that $f$ \emph{extends ideals} if for every ideal $I\subseteq B$,
\[
I = f(f^{-1}(I))B,
\]
where $f(f^{-1}(I))B$ denotes the ideal of $B$ generated by the image of $f^{-1}(I)$ under $f$.
\end{definition}

\begin{proposition}\label{prop:composition-extends-ideals}
Let $f\colon A\to B$ and $g\colon B\to C$ be morphisms of commutative rings. If $f$ and $g$ extend ideals, then the composition
\[
g\circ f\colon A\longrightarrow C
\]
also extends ideals.
\end{proposition}

\begin{proof}
Let $K\subseteq C$ be an ideal. We must show that
\[
K = (g\circ f)\big((g\circ f)^{-1}(K)\big)C.
\]
Set $J := g^{-1}(K)$, which is an ideal of $B$. Then
\[
(g\circ f)^{-1}(K) = f^{-1}(g^{-1}(K)) = f^{-1}(J).
\]
Since $g$ extends ideals, we have
\[
K = g(g^{-1}(K))C = g(J)C.
\]
Since $f$ extends ideals, we have
\[
J = f(f^{-1}(J))B.
\]
Therefore,
\[
K = g(J)C = g\big(f(f^{-1}(J))B\big)C
= (g\circ f)\big(f^{-1}(J)\big)C
= (g\circ f)\big((g\circ f)^{-1}(K)\big)C.
\]
Hence $g\circ f$ extends ideals.
\end{proof}

The identity morphism $\mathrm{id}_A\colon A\to A$ extends ideals: for every ideal $I\subseteq A$,
\[
I = \mathrm{id}_A(\mathrm{id}_A^{-1}(I))A = I.
\]
By Proposition~\ref{prop:composition-extends-ideals}, the composition of two morphisms that extend ideals again extends ideals. Therefore, commutative rings together with ring morphisms that extend ideals form a category. Since this category has the same objects as the category of commutative rings, and its morphisms are a subclass of all ring morphisms containing all identities and closed under composition, it is a \emph{wide subcategory} of the category of commutative rings.

We next introduce the class of morphisms for which the induced map on coprime spectra will be defined.

\begin{definition}\label{def:cf-morphism}
Let $f\colon A\to B$ be a morphism of commutative rings that extends ideals. We say that $f$ is \emph{coprime faithful} if for every $C\in\operatorname{CoSp}(A)$, the ideal $f(C)B$ is a non-zero ideal of $B$. A morphism that extends ideals and is coprime faithful is called a \emph{CF-morphism}.
\end{definition}

The following lemma is the key technical ingredient relating the coproduct to inverse images.

\begin{lemma}\label{lem:preimage-coproduct}
Let $f\colon A\to B$ be a morphism of commutative rings that extends ideals. Then for every pair of ideals $I,J\subseteq B$,
\[
f^{-1}([I:J]) \subseteq [f^{-1}(I):f^{-1}(J)].
\]
\end{lemma}

\begin{proof}
Let $x\in f^{-1}([I:J])$. Then $f(x)\in [I:J]=(J:\ann_B(I))$, so
\[
f(x)\ann_B(I)\subseteq J.
\]
We must show that $x\in [f^{-1}(I):f^{-1}(J)]$. By definition,
\[
[f^{-1}(I):f^{-1}(J)] = (f^{-1}(J):\ann_A(f^{-1}(I))),
\]
so it suffices to prove that
\[
x\ann_A(f^{-1}(I))\subseteq f^{-1}(J).
\]
Let $y\in\ann_A(f^{-1}(I))$. Then $y f^{-1}(I)=0$. Applying $f$, we get
\[
f(y)f(f^{-1}(I))=0.
\]
Since $f$ extends ideals, $I=f(f^{-1}(I))B$, hence $f(y)I=0$. Therefore $f(y)\in\ann_B(I)$. Consequently,
\[
f(xy)=f(x)f(y)\in f(x)\ann_B(I)\subseteq J,
\]
which means $xy\in f^{-1}(J)$. Thus $x\in [f^{-1}(I):f^{-1}(J)]$, as required.
\end{proof}

\begin{theorem}\label{thm:image-coprime-extends-ideals}
Let $f\colon A\to B$ be a morphism of commutative rings that extends ideals, and let $C\in\operatorname{CoSp}(A)$ satisfy $f(C)B\neq 0$. Then $f(C)B\in\operatorname{CoSp}(B)$.
\end{theorem}

\begin{proof}
Set $D:=f(C)B$. By hypothesis, $D\neq 0$. Let $I,J\subseteq B$ be ideals such that
\[
D\subseteq [I:J].
\]
Since $f(C)\subseteq D$, we have $f(C)\subseteq [I:J]$. Hence
\[
C\subseteq f^{-1}([I:J]).
\]
By Lemma~\ref{lem:preimage-coproduct},
\[
f^{-1}([I:J])\subseteq [f^{-1}(I):f^{-1}(J)].
\]
Therefore,
\[
C\subseteq [f^{-1}(I):f^{-1}(J)].
\]
Since $C$ is a coprime ideal of $A$, it follows that
\[
C\subseteq f^{-1}(I)
\quad\text{or}\quad
C\subseteq f^{-1}(J).
\]
Applying $f$ gives
\[
f(C)\subseteq I
\quad\text{or}\quad
f(C)\subseteq J.
\]
Consequently,
\[
D=f(C)B\subseteq I
\quad\text{or}\quad
D\subseteq J.
\]
Thus $D$ satisfies the defining condition of a coprime ideal, and since $D\neq 0$, we conclude that $D\in\operatorname{CoSp}(B)$.
\end{proof}

\begin{corollary}\label{cor:cf-image-coprime}
If $f\colon A\to B$ is a CF-morphism, then for every $C\in\operatorname{CoSp}(A)$, the ideal $f(C)B$ belongs to $\operatorname{CoSp}(B)$.
\end{corollary}

\begin{proof}
By definition of CF-morphism, $f$ extends ideals and $f(C)B\neq 0$ for every $C\in\operatorname{CoSp}(A)$. Theorem~\ref{thm:image-coprime-extends-ideals} then gives $f(C)B\in\operatorname{CoSp}(B)$.
\end{proof}

We now verify that the class of morphisms extending ideals is rich enough for our purposes: it contains localizations and surjections.

\begin{proposition}\label{prop:localization-extends-ideals}
Let $A$ be a commutative ring, $S\subseteq A$ a multiplicative set, and let
\[
\varphi\colon A\longrightarrow S^{-1}A
\]
be the localization morphism. Then $\varphi$ extends ideals. That is, for every ideal $I\subseteq S^{-1}A$,
\[
I = \varphi(\varphi^{-1}(I))\,S^{-1}A.
\]
\end{proposition}

\begin{proof}
Let $I\subseteq S^{-1}A$ be an ideal. Since $\varphi(\varphi^{-1}(I))\subseteq I$ and $I$ is an ideal of $S^{-1}A$, we have
\[
\varphi(\varphi^{-1}(I))\,S^{-1}A \subseteq I.
\]
Conversely, let $x\in I$. Write $x=a/s$ with $a\in A$ and $s\in S$. Since $I$ is an ideal of $S^{-1}A$, we have
\[
s\cdot \frac{a}{s} = \frac{a}{1} = \varphi(a) \in I.
\]
Hence $a\in \varphi^{-1}(I)$, and therefore
\[
x = \frac{a}{s} \in S^{-1}(\varphi^{-1}(I)) = \varphi(\varphi^{-1}(I))\,S^{-1}A.
\]
Thus $I\subseteq \varphi(\varphi^{-1}(I))\,S^{-1}A$, and the reverse inclusion was already shown. Consequently,
\[
I = \varphi(\varphi^{-1}(I))\,S^{-1}A,
\]
so $\varphi$ extends ideals.
\end{proof}

\begin{proposition}\label{prop:projections-extend-ideals}
Every surjective morphism of commutative rings $f\colon A\to B$ extends ideals. In particular, every canonical projection
\[
\pi\colon A\longrightarrow A/I
\]
extends ideals.
\end{proposition}

\begin{proof}
Let $J\subseteq B$ be an ideal. Since $f$ is surjective,
\[
f(f^{-1}(J))=J.
\]
Therefore,
\[
f(f^{-1}(J))B = JB = J,
\]
because $J$ is an ideal of $B$. Hence $f$ extends ideals.

The canonical projection $\pi\colon A\to A/I$ is surjective, so by the same argument it also extends ideals.
\end{proof}

We can now define the induced map on coprime spectra.

\begin{definition}\label{def:cosp-functor}
Let $f\colon A\to B$ be a CF-morphism. Define
\[
\operatorname{CoSp}(f)\colon \operatorname{CoSp}(A)\longrightarrow \operatorname{CoSp}(B),
\qquad C\longmapsto f(C)B.
\]
\end{definition}

\begin{proposition}\label{prop:cosp-continuous}
Let $f\colon A\to B$ be a CF-morphism. Then the induced map $\operatorname{CoSp}(f)\colon \operatorname{CoSp}(A)\to \operatorname{CoSp}(B)$ is continuous.
\end{proposition}

\begin{proof}
Let $I\subseteq B$ be an ideal. We claim that
\[
\operatorname{CoSp}(f)^{-1}(V_c(I))=V_c(f^{-1}(I)).
\]
For $C\in\operatorname{CoSp}(A)$,
\[
C\in \operatorname{CoSp}(f)^{-1}(V_c(I))
\iff f(C)B\subseteq I
\iff C\subseteq f^{-1}(I)
\iff C\in V_c(f^{-1}(I)).
\]
Thus the inverse image of every closed set is closed, so $\operatorname{CoSp}(f)$ is continuous.
\end{proof}

Finally, we record the functoriality of this construction.

The assignment
\[
A\longmapsto \operatorname{CoSp}(A),\qquad f\longmapsto \operatorname{CoSp}(f)
\]
defines a covariant functor from the category of commutative rings with CF-morphisms to the category of topological spaces with continuous maps. Indeed, the identity map induces the identity on spectra. If $f\colon A\to B$ and $g\colon B\to D$ are CF-morphisms, then $g\circ f$ is a CF-morphism: composition of morphisms extending ideals extends ideals, and for $C\in\operatorname{CoSp}(A)$, the ideal $f(C)B$ is coprime in $B$, so $g(f(C)B)D$ is coprime in $D$. But
\[
g(f(C)B)D=(g\circ f)(C)D,
\]
so the composition is preserved.

\section{Dual Rings}
\label{sec:dual-rings}

In this section we focus on rings for which the double annihilator of every ideal returns the ideal itself. Such rings are called dual rings, and they provide a setting where the coprime spectrum and the prime spectrum are homeomorphic.

\begin{definition}\label{def:dual-ring}
A ring $A$ is called a \emph{dual ring} if
\[
\ann(\ann(I)) = I
\]
for every ideal $I \subseteq A$. The classical references for dual rings are \cite{hall1939type}, \cite{baer1943rings}, and \cite{kaplansky1948dual}.
\end{definition}

The notion of a dual ring is intimately related to the classical concept of a quasi-Frobenius ring. Recall that a commutative ring $A$ is called \emph{quasi-Frobenius} if it is Artinian and self-injective. A fundamental characterization, due to Morita and others, states that a commutative ring is quasi-Frobenius if and only if it is Noetherian and satisfies the double annihilator condition $\ann(\ann(I)) = I$ for every ideal $I$ of $A$ (see \cite{lam1999lectures}). In particular, every commutative quasi-Frobenius ring is a dual ring in the sense of Definition~\ref{def:dual-ring}. Conversely, a commutative Noetherian dual ring is quasi-Frobenius, so the two notions coincide precisely for commutative Noetherian rings.

A rich source of examples arises from group algebras. Let $K$ be a field and $G$ a finite abelian group. The group algebra $K[G]$ is commutative, finite-dimensional over $K$, and it is well known that every group algebra of a finite group over a field is a Frobenius algebra \cite{curtis1962representation}. Since every Frobenius algebra is quasi-Frobenius, it follows that $K[G]$ is a dual ring. In these rings, the annihilator is typically non-zero, and the coproduct $[I:J]$ is highly non-trivial, providing a fertile ground for the theory developed here. More generally, a group ring $A[G]$ is quasi-Frobenius if and only if $G$ is finite and $A$ is quasi-Frobenius; in the commutative case, this reduces to $G$ abelian and $A$ commutative quasi-Frobenius \cite{nicholson2003quasi}.

We now prove three fundamental results that hold in any dual ring.

\begin{proposition}\label{prop:dual-annihilator-product}
Let $A$ be a dual ring and let $I,J\subseteq A$ be ideals. Then
\[
\ann(IJ) = [\ann(I):\ann(J)].
\]
Consequently, taking annihilators on both sides and using the dual property, we also have
\[
IJ = \ann\!\big([\ann(I):\ann(J)]\big).
\]
\end{proposition}

\begin{proof}
Using Proposition~\ref{pro:annihilator}.(5) and the definition of the coproduct, we compute:
\[
[\ann(I):\ann(J)] = (\ann(J) : \ann(\ann(I))).
\]
Since $A$ is dual, $\ann(\ann(I)) = I$. Hence
\[
[\ann(I):\ann(J)] = (\ann(J) : I).
\]
Again by Proposition~\ref{pro:annihilator}.(5), $(\ann(J) : I) = \ann(IJ)$. Therefore,
\[
[\ann(I):\ann(J)] = \ann(IJ),
\]
as claimed. The second identity follows by applying the double annihilator to both sides.
\end{proof}

The next result shows that, in a dual ring, the coprimality of an ideal is equivalent to the primality of its annihilator. This establishes a perfect duality between the two spectra.

\begin{theorem}\label{thm:dual-coprime-prime}
Let $A$ be a dual ring and $C \subseteq A$ a non-zero ideal. Then $C$ is coprime if and only if $\ann(C)$ is a prime ideal.
\end{theorem}

\begin{proof}
$(\Rightarrow)$ Suppose $C$ is coprime. Since $C \neq 0$, we have $\ann(C) \neq A$, so $\ann(C)$ is a proper ideal. Let $P, Q \subseteq A$ such that
\[
PQ \subseteq \ann(C).
\]
Set $I = \ann(P)$ and $J = \ann(Q)$. Since $A$ is dual, $\ann(I) = P$ and $\ann(J) = Q$.

We claim that $C \subseteq [I:J]$. Indeed,
\[
[I:J] = (J:\ann(I)) = (J:P).
\]
Thus it suffices to show $CP \subseteq J = \ann(Q)$. But
\[
C P Q \subseteq C \cdot \ann(C) = 0,
\]
so $CP \subseteq \ann(Q) = J$. Therefore $C \subseteq (J:P) = [I:J]$.

Since $C$ is coprime, we have
\[
C \subseteq I \quad \text{or} \quad C \subseteq J.
\]
If $C \subseteq I = \ann(P)$, taking annihilators and using duality gives
\[
P = \ann(\ann(P)) = \ann(I) \subseteq \ann(C).
\]
Similarly, if $C \subseteq J = \ann(Q)$, then $Q \subseteq \ann(C)$. Hence
\[
P \subseteq \ann(C) \quad \text{or} \quad Q \subseteq \ann(C).
\]
As $P$ and $Q$ were arbitrary, $\ann(C)$ is prime.

$(\Leftarrow)$ Conversely, suppose $P := \ann(C)$ is a prime ideal. Let $I, J \subseteq A$ be ideals such that
\[
C \subseteq [I:J].
\]
By definition, $C \cdot \ann(I) \subseteq J$. Taking annihilators on both sides and using Proposition~\ref{pro:annihilator}.(5), we obtain
\[
\ann(J) \subseteq \ann(C \cdot \ann(I)) = (\ann(C):\ann(I)) = (P:\ann(I)).
\]
The inclusion $\ann(J) \subseteq (P:\ann(I))$ is equivalent to
\[
\ann(J) \cdot \ann(I) \subseteq P.
\]
Since $P$ is prime, it follows that
\[
\ann(I) \subseteq P \quad \text{or} \quad \ann(J) \subseteq P.
\]
If $\ann(I) \subseteq P = \ann(C)$, taking annihilators and using duality gives
\[
C = \ann(P) \subseteq \ann(\ann(I)) = I.
\]
Similarly, if $\ann(J) \subseteq P$, then $C \subseteq J$. Therefore $C \subseteq I$ or $C \subseteq J$, which proves that $C$ is coprime.
\end{proof}

The following corollary is an immediate consequence of the theorem and shows that the annihilator map gives a bijection between the coprime spectrum and the prime spectrum.

\begin{corollary}\label{cor:dual-coprime-prime-bijection}
Let $A$ be a dual ring. The map
\[
\ann\colon \operatorname{CoSp}(A) \longrightarrow \operatorname{Spec}(A), \qquad C \longmapsto \ann(C)
\]
is a bijection, with inverse given by $P \mapsto \ann(P)$.
\end{corollary}

Finally, we prove that this bijection is actually a homeomorphism between the topological spaces $\operatorname{CoSp}(A)$ and $\operatorname{Spec}(A)$, where the latter is endowed with the Zariski topology and the former with the topology defined in Section~\ref{sec:coprime-spectrum}.

\begin{theorem}\label{thm:dual-homeomorphism}
Let $A$ be a commutative dual ring. Then the map
\[
\ann\colon \operatorname{CoSp}(A) \longrightarrow \operatorname{Spec}(A), \qquad C \longmapsto \ann(C)
\]
is a homeomorphism.
\end{theorem}

\begin{proof}
By Corollary~\ref{cor:dual-coprime-prime-bijection}, the map is a bijection. It remains to show that both $\ann$ and its inverse are continuous.

First, we show that $\ann$ is continuous. Let $V(P) \subseteq \operatorname{Spec}(A)$ be a closed set, where $P \subseteq A$ is an ideal. Then
\[
\ann^{-1}(V(P)) = \{\, C \in \operatorname{CoSp}(A) \mid \ann(C) \in V(P) \,\}.
\]
Now $\ann(C) \in V(P)$ means $P \subseteq \ann(C)$. Since $A$ is a dual ring, taking annihilators on both sides yields
\[
P \subseteq \ann(C) \iff \ann(\ann(C)) \subseteq \ann(P) \iff C \subseteq \ann(P).
\]
Therefore,
\[
\ann^{-1}(V(P)) = \{\, C \in \operatorname{CoSp}(A) \mid C \subseteq \ann(P) \,\} = V_c(\ann(P)),
\]
which is a closed set in $\operatorname{CoSp}(A)$. Hence $\ann$ is continuous.

Now we show that the inverse map $\ann^{-1}\colon \operatorname{Spec}(A) \to \operatorname{CoSp}(A)$, given by $P \mapsto \ann(P)$, is continuous. Let $V_c(I) \subseteq \operatorname{CoSp}(A)$ be a closed set, where $I \subseteq A$. Then
\[
(\ann^{-1})^{-1}(V_c(I)) = \{\, P \in \operatorname{Spec}(A) \mid \ann(P) \in V_c(I) \,\}.
\]
Now $\ann(P) \in V_c(I)$ means $\ann(P) \subseteq I$. Since $A$ is a dual ring, taking annihilators gives
\[
\ann(P) \subseteq I \iff \ann(I) \subseteq \ann(\ann(P)) \iff \ann(I) \subseteq P.
\]
Therefore,
\[
(\ann^{-1})^{-1}(V_c(I)) = \{\, P \in \operatorname{Spec}(A) \mid \ann(I) \subseteq P \,\} = V(\ann(I)),
\]
which is a closed set in $\operatorname{Spec}(A)$. Hence the inverse map is continuous.

Thus $\ann$ is a bijective continuous map with continuous inverse, so it is a homeomorphism.
\end{proof}

Theorem~\ref{thm:dual-homeomorphism} establishes a strong duality between the coprime spectrum and the prime spectrum in dual rings. In particular, the topological spaces $\operatorname{CoSp}(A)$ and $\operatorname{Spec}(A)$ are homeomorphic, so any topological property of the prime spectrum (such as compactness, irreducibility, or the structure of closed sets) transfers directly to the coprime spectrum in the dual ring setting. This generalises the classical correspondence between minimal ideals and maximal ideals in quasi-Frobenius rings.


\section*{Acknowledgments}
The authors are grateful to A. Sannai for a careful reading of an earlier version of this paper and for useful suggestions.

\bibliographystyle{amsplain}
\bibliography{biblio}

\end{document}